\documentclass[12pt,leqno]{article}
\usepackage{amsfonts}
\usepackage{amssymb}
\usepackage{amsmath}
\usepackage{indentfirst}

\newcommand{\be}{\begin{equation}}
\newcommand{\ee}{\end{equation}}
\newcommand{\bea}{\begin{eqnarray*}}
\newcommand{\eea}{\end{eqnarray*}}
\newcommand{\ba}{\begin{array}}
\newcommand{\ea}{\end{array}}
\newcommand{\bi}{\begin{itemize}}
\newcommand{\ei}{\end{itemize}}
\newcommand{\bc}{\begin{center}}
\newcommand{\ec}{\end{center}}
\newcommand{\bfr}{\begin{flushright}}
\newcommand{\efr}{\end{flushright}}

\input{tcilatex}
\begin{document}

\title{Smarandache's Cevian Triangle Theorem in The Einstein Relativistic
Velocity Model of Hyperbolic Geometry}
\author{\textbf{C\u{a}t\u{a}lin Barbu}}
\date{{\small "Vasile Alecsandri" College - Bac\u{a}u, str. Iosif Cocea, nr.
12, sc. A, ap. 13, Romania}\\
$\ \ \ $\ {\small kafka\_mate@yahoo.com}}
\maketitle

\begin{abstract}
In this note, we present a proof of Smarandache's cevian triangle hyperbolic
theorem in the Einstein relativistic velocity model of hyperbolic geometry.
\end{abstract}

\bigskip 

\textbf{2000 Mathematical Subject Classification}: 51K05, 51M10, 30F45,
20N99, 51B10

\textbf{Keywords and phrases}: hyperbolic geometry, hyperbolic triangle,
Smarandache's cevian triangle, gyrovector, Einstein relativistic velocity
model

\bigskip

\textbf{1. \ Introduction}

Hyperbolic geometry appeared in the first half of the $19^{th}$ century as
an attempt to understand Euclid's axiomatic basis for geometry. It is also
known as a type of non-Euclidean geometry, being in many respects similar to
Euclidean geometry. Hyperbolic geometry includes such concepts as: distance,
angle and both of them have many theorems in common.There are known many
main models for hyperbolic geometry, such as: Poincar\'{e} disc model,
Poincar\'{e} half-plane, Klein model, Einstein relativistic velocity model,
etc. The hyperbolic geometry is a non-Euclidian geometry. Here, in this
study, we present a proof of Smarandache's cevian triangle hyperbolic
theorem in the Einstein relativistic velocity model of hyperbolic geometry.
Smarandache's cevian triangle theorem states that if $A_{1}B_{1}C_{1}$ is
the cevian triangle of point $P$ with respect to the triangle $ABC,$ then $%
\frac{PA}{PA_{1}}\cdot \frac{PB}{PB_{1}}\cdot \frac{PC}{PC_{1}}=\frac{%
AB\cdot BC\cdot CA}{A_{1}B\cdot B_{1}C\cdot C_{1}A}$ [1].

Let $D$ denote the complex unit disc in complex $z$ - plane, i.e.%
\begin{equation*}
D=\{z\in 
\mathbb{C}
:\left\vert z\right\vert <1\}. 
\end{equation*}%
The most general M\"{o}bius transformation of $D$ is%
\begin{equation*}
z\rightarrow e^{i\theta }\frac{z_{0}+z}{1+\overline{z_{0}}z}=e^{i\theta
}(z_{0}\oplus z), 
\end{equation*}%
which induces the M\"{o}bius addition $\oplus $ in $D$, allowing the M\"{o}%
bius transformation of the disc to be viewed as a M\"{o}bius left
gyrotranslation%
\begin{equation*}
z\rightarrow z_{0}\oplus z=\frac{z_{0}+z}{1+\overline{z_{0}}z} 
\end{equation*}%
followed by a rotation. Here $\theta \in 
\mathbb{R}
$ is a real number, $z,z_{0}\in D,$ and $\overline{z_{0}}$ is the complex
conjugate of $z_{0}.$ Let $Aut(D,\oplus )$ be the automorphism group of the
grupoid $(D,\oplus )$. If we define%
\begin{equation*}
gyr:D\times D\rightarrow Aut(D,\oplus ),gyr[a,b]=\frac{a\oplus b}{b\oplus a}=%
\frac{1+a\overline{b}}{1+\overline{a}b}, 
\end{equation*}%
then is true gyrocommutative law%
\begin{equation*}
a\oplus b=gyr[a,b](b\oplus a). 
\end{equation*}

A gyrovector space $(G,\oplus ,\otimes )$ is a gyrocommutative gyrogroup $%
(G,\oplus )$ that obeys the following axioms:

\textit{(1) }$gyr[\mathbf{u,v]a\cdot }$ $gyr[\mathbf{u,v]b=a\cdot b}$\ \ for
all points $\mathbf{a,b,u,v\in }G.$

\textit{(2) }$G$ admits a scalar multiplication, $\otimes $, possessing the
following properties. For all real numbers $r,r_{1},r_{2}\in 
\mathbb{R}
$ and all points $\mathbf{a\in }G$:

\textit{(G1) }$1\otimes \mathbf{a=a}$

\textit{(G2) }$(r_{1}+r_{2})\otimes \mathbf{a}=r_{1}\otimes \mathbf{a}\oplus
r_{2}\otimes \mathbf{a}$

\textit{(G3) }$(r_{1}r_{2})\otimes \mathbf{a}=r_{1}\otimes (r_{2}\otimes 
\mathbf{a)}$

\textit{(G4) }$\frac{\left\vert r\right\vert \otimes \mathbf{a}}{\left\Vert
r\otimes \mathbf{a}\right\Vert }=\frac{\mathbf{a}}{\left\Vert \mathbf{a}%
\right\Vert }$

\textit{(G5) }$gyr[\mathbf{u,v}](r\otimes \mathbf{a)}=r\otimes gyr[\mathbf{%
u,v}]\mathbf{a}$

\textit{(G6) }$gyr[r_{1}\otimes \mathbf{v},r_{1}\otimes \mathbf{v]=}\mathit{1%
}$

\textit{(3) }Real vector space structure $(\left\Vert G\right\Vert ,\oplus
,\otimes )$ for the set $\left\Vert G\right\Vert $ of onedimensional
"vectors"%
\begin{equation*}
\left\Vert G\right\Vert =\{\pm \left\Vert \mathbf{a}\right\Vert :\mathbf{a}%
\in G\}\subset 
\mathbb{R}
\end{equation*}%
with vector addition $\oplus $ and scalar multiplication $\otimes ,$ such
that for all $r\in 
\mathbb{R}
$ and $\mathbf{a,b}\in G,$

\textit{(G7) }$\left\Vert r\otimes \mathbf{a}\right\Vert =\left\vert
r\right\vert \otimes \left\Vert \mathbf{a}\right\Vert $

\textit{(G8) }$\left\Vert \mathbf{a}\oplus \mathbf{b}\right\Vert \leq
\left\Vert \mathbf{a}\right\Vert \oplus \left\Vert \mathbf{b}\right\Vert $

\begin{theorem}
\textbf{(The Hyperbolic Theorem of Ceva in Einstein Gyrovector Space) }Let $%
\mathbf{a}_{1},\mathbf{a}_{2},$ and $\mathbf{a}_{3}$ be three
non-gyrocollinear points in an Einstein gyrovector space $(V_{s},\oplus
,\otimes ).$ Furthermore, let $\mathbf{a}_{123}$ be a point in their
gyroplane, which is off the gyrolines $\mathbf{a}_{1}\mathbf{a}_{2},\mathbf{a%
}_{2}\mathbf{a}_{3},$ and $\mathbf{a}_{3}\mathbf{a}_{1}.$ If $\mathbf{a}_{1}%
\mathbf{a}_{123}$ meets $\mathbf{a}_{2}\mathbf{a}_{3}$ at $\mathbf{a}_{23},$
etc., then 
\begin{equation*}
\frac{\gamma _{\ominus \mathbf{a}_{1}\oplus \mathbf{a}_{12}}\left\Vert
\ominus \mathbf{a}_{1}\oplus \mathbf{a}_{12}\right\Vert }{\gamma _{\ominus 
\mathbf{a}_{2}\oplus \mathbf{a}_{12}}\left\Vert \ominus \mathbf{a}_{2}\oplus 
\mathbf{a}_{12}\right\Vert }\frac{\gamma _{\ominus \mathbf{a}_{2}\oplus 
\mathbf{a}_{23}}\left\Vert \ominus \mathbf{a}_{2}\oplus \mathbf{a}%
_{23}\right\Vert }{\gamma _{\ominus \mathbf{a}_{3}\oplus \mathbf{a}%
_{23}}\left\Vert \ominus \mathbf{a}_{3}\oplus \mathbf{a}_{23}\right\Vert }%
\frac{\gamma _{\ominus \mathbf{a}_{3}\oplus \mathbf{a}_{13}}\left\Vert
\ominus \mathbf{a}_{3}\oplus \mathbf{a}_{13}\right\Vert }{\gamma _{\ominus 
\mathbf{a}_{1}\oplus \mathbf{a}_{13}}\left\Vert \ominus \mathbf{a}_{1}\oplus 
\mathbf{a}_{13}\right\Vert }=1, 
\end{equation*}%
(here $\gamma _{\mathbf{v}}=\frac{1}{\sqrt{1-\frac{\left\Vert \mathbf{v}%
\right\Vert ^{2}}{s^{2}}}}$ is the gamma factor).
\end{theorem}

(see [2, pp 461])

\begin{theorem}
\textbf{(The Hyperbolic Theorem of Menelaus in Einstein Gyrovector Space) }%
Let $\mathbf{a}_{1},\mathbf{a}_{2},$ and $\mathbf{a}_{3}$ be three
non-gyrocollinear points in an Einstein gyrovector space $(V_{s},\oplus
,\otimes ).$ If a gyroline meets the sides of gyrotriangle $\mathbf{a}_{1}%
\mathbf{a}_{2}\mathbf{a}_{3}$\textbf{\ }at points $\mathbf{a}_{12},\mathbf{a}%
_{13},\mathbf{a}_{23,}$ then 
\begin{equation*}
\frac{\gamma _{\ominus \mathbf{a}_{1}\oplus \mathbf{a}_{12}}\left\Vert
\ominus \mathbf{a}_{1}\oplus \mathbf{a}_{12}\right\Vert }{\gamma _{\ominus 
\mathbf{a}_{2}\oplus \mathbf{a}_{12}}\left\Vert \ominus \mathbf{a}_{2}\oplus 
\mathbf{a}_{12}\right\Vert }\frac{\gamma _{\ominus \mathbf{a}_{2}\oplus 
\mathbf{a}_{23}}\left\Vert \ominus \mathbf{a}_{2}\oplus \mathbf{a}%
_{23}\right\Vert }{\gamma _{\ominus \mathbf{a}_{3}\oplus \mathbf{a}%
_{23}}\left\Vert \ominus \mathbf{a}_{3}\oplus \mathbf{a}_{23}\right\Vert }%
\frac{\gamma _{\ominus \mathbf{a}_{3}\oplus \mathbf{a}_{13}}\left\Vert
\ominus \mathbf{a}_{3}\oplus \mathbf{a}_{13}\right\Vert }{\gamma _{\ominus 
\mathbf{a}_{1}\oplus \mathbf{a}_{13}}\left\Vert \ominus \mathbf{a}_{1}\oplus 
\mathbf{a}_{13}\right\Vert }=1 
\end{equation*}
\end{theorem}

(see [2, pp 463])

For further details we refer to the recent book of A.Ungar [2].

\bigskip

\textbf{2. Main result}

\bigskip

\bigskip In this section, we present a proof of Smarandache's cevian
triangle hyperbolic theorem in the Einstein relativistic velocity model of
hyperbolic geometry.

\begin{theorem}
If $A_{1}B_{1}C_{1}$ is the cevian gyrotriangle of gyropoint $P$ with
respect to the gyrotriangle $ABC,$ then%
\begin{equation*}
\frac{\gamma _{_{\left\vert PA\right\vert }\left\vert PA\right\vert }}{%
\gamma _{_{\left\vert PA_{1}\right\vert }\left\vert PA_{1}\right\vert }}%
\cdot \frac{\gamma _{_{\left\vert PB\right\vert }\left\vert PB\right\vert }}{%
\gamma _{_{\left\vert PB_{1}\right\vert }\left\vert PB_{1}\right\vert }}%
\cdot \frac{\gamma _{_{\left\vert PC\right\vert }\left\vert PC\right\vert }}{%
\gamma _{_{\left\vert PC_{1}\right\vert }\left\vert PC_{1}\right\vert }}=%
\frac{\gamma _{_{\left\vert AB\right\vert }\left\vert AB\right\vert }\cdot
\gamma _{_{\left\vert BC\right\vert }\left\vert BC\right\vert }\cdot \gamma
_{_{\left\vert CA\right\vert }\left\vert CA\right\vert }}{\gamma
_{_{\left\vert AB_{1}\right\vert }\left\vert AB_{1}\right\vert }\cdot \gamma
_{_{\left\vert BC_{1}\right\vert }\left\vert BC_{1}\right\vert }\cdot \gamma
_{_{\left\vert CA_{1}\right\vert }\left\vert CA_{1}\right\vert }}. 
\end{equation*}
\end{theorem}

\begin{proof}
If we use a theorem 2 in the gyrotriangle $ABC$ (see Figure), we have%
\begin{equation}
\gamma _{_{\left\vert AC_{1}\right\vert }\left\vert AC_{1}\right\vert }\cdot
\gamma _{_{\left\vert BA_{1}\right\vert }\left\vert BA_{1}\right\vert }\cdot
\gamma _{_{\left\vert CB_{1}\right\vert }\left\vert CB_{1}\right\vert
}=\gamma _{_{\left\vert AB_{1}\right\vert }\left\vert AB_{1}\right\vert
}\cdot \gamma _{_{\left\vert BC_{1}\right\vert }\left\vert BC_{1}\right\vert
}\cdot \gamma _{_{\left\vert CA_{1}\right\vert }\left\vert CA_{1}\right\vert
}
\end{equation}%
If we use a theorem 1 in the gyrotriangle $AA_{1}B,$ cut by the gyroline $%
CC_{1},$ we get 
\begin{equation}
\gamma _{_{\left\vert AC_{1}\right\vert }\left\vert AC_{1}\right\vert }\cdot
\gamma _{_{\left\vert BC\right\vert }\left\vert BC\right\vert }\cdot \gamma
_{_{\left\vert A_{1}P\right\vert }\left\vert A_{1}P\right\vert }=\gamma
_{_{\left\vert AP\right\vert }\left\vert AP\right\vert }\cdot \gamma
_{_{\left\vert A_{1}C\right\vert }\left\vert A_{1}C\right\vert }\cdot \gamma
_{_{\left\vert BC_{1}\right\vert }\left\vert BC_{1}\right\vert }.
\end{equation}%
If we use a theorem 1 in the gyrotriangle $BB_{1}C,$ cut by the gyroline $%
AA_{1},$ we get 
\begin{equation}
\gamma _{_{\left\vert BA_{1}\right\vert }\left\vert BA_{1}\right\vert }\cdot
\gamma _{_{\left\vert CA\right\vert }\left\vert CA\right\vert }\cdot \gamma
_{_{\left\vert B_{1}P\right\vert }\left\vert B_{1}P\right\vert }=\gamma
_{_{\left\vert BP\right\vert }\left\vert BP\right\vert }\cdot \gamma
_{_{\left\vert B_{1}A\right\vert }\left\vert B_{1}A\right\vert }\cdot \gamma
_{_{\left\vert CA_{1}\right\vert }\left\vert CA_{1}\right\vert }.
\end{equation}%
If we use a theorem 1 in the gyrotriangle $CC_{1}A,$ cut by the gyroline $%
BB_{1},$ we get 
\begin{equation}
\gamma _{_{\left\vert CB_{1}\right\vert }\left\vert CB_{1}\right\vert }\cdot
\gamma _{_{\left\vert AB\right\vert }\left\vert AB\right\vert }\cdot \gamma
_{_{\left\vert C_{1}P\right\vert }\left\vert C_{1}P\right\vert }=\gamma
_{_{\left\vert CP\right\vert }\left\vert CP\right\vert }\cdot \gamma
_{_{\left\vert C_{1}B\right\vert }\left\vert C_{1}B\right\vert }\cdot \gamma
_{_{\left\vert AB_{1}\right\vert }\left\vert AB_{1}\right\vert }.
\end{equation}%
We divide each relation (2), (3), and (4) by relation (1), and we obtain%
\begin{equation}
\frac{\gamma _{_{\left\vert PA\right\vert }\left\vert PA\right\vert }}{%
\gamma _{_{\left\vert PA_{1}\right\vert }\left\vert PA_{1}\right\vert }}=%
\frac{\gamma _{_{\left\vert BC\right\vert }\left\vert BC\right\vert }}{%
\gamma _{_{\left\vert BA_{1}\right\vert }\left\vert BA_{1}\right\vert }}%
\cdot \frac{\gamma _{_{\left\vert B_{1}A\right\vert }\left\vert
B_{1}A\right\vert }}{\gamma _{_{\left\vert B_{1}C\right\vert }\left\vert
B_{1}C\right\vert }},
\end{equation}

\begin{equation}
\frac{\gamma _{_{\left\vert PB\right\vert }\left\vert PB\right\vert }}{%
\gamma _{_{\left\vert PB_{1}\right\vert }\left\vert PB_{1}\right\vert }}=%
\frac{\gamma _{_{\left\vert CA\right\vert }\left\vert CA\right\vert }}{%
\gamma _{_{\left\vert CB_{1}\right\vert }\left\vert CB_{1}\right\vert }}%
\cdot \frac{\gamma _{_{\left\vert C_{1}B\right\vert }\left\vert
C_{1}B\right\vert }}{\gamma _{_{\left\vert C_{1}A\right\vert }\left\vert
C_{1}A\right\vert }},
\end{equation}

\begin{equation}
\frac{\gamma _{_{\left\vert PC\right\vert }\left\vert PC\right\vert }}{%
\gamma _{_{\left\vert PC_{1}\right\vert }\left\vert PC_{1}\right\vert }}=%
\frac{\gamma _{_{\left\vert AB\right\vert }\left\vert AB\right\vert }}{%
\gamma _{_{\left\vert AC_{1}\right\vert }\left\vert AC_{1}\right\vert }}%
\cdot \frac{\gamma _{_{\left\vert A_{1}C\right\vert }\left\vert
A_{1}C\right\vert }}{\gamma _{_{\left\vert A_{1}B\right\vert }\left\vert
A_{1}B\right\vert }}.
\end{equation}%
Multiplying (5) by (6) and by (7), we have 
\begin{equation*}
\frac{\gamma _{_{\left\vert PA\right\vert }\left\vert PA\right\vert }}{%
\gamma _{_{\left\vert PA_{1}\right\vert }\left\vert PA_{1}\right\vert }}%
\cdot \frac{\gamma _{_{\left\vert PB\right\vert }\left\vert PB\right\vert }}{%
\gamma _{_{\left\vert PB_{1}\right\vert }\left\vert PB_{1}\right\vert }}%
\cdot \frac{\gamma _{_{\left\vert PC\right\vert }\left\vert PC\right\vert }}{%
\gamma _{_{\left\vert PC_{1}\right\vert }\left\vert PC_{1}\right\vert }}= 
\end{equation*}%
\begin{equation}
\frac{\gamma _{_{\left\vert AB\right\vert }\left\vert AB\right\vert }\cdot
\gamma _{_{\left\vert BC\right\vert }\left\vert BC\right\vert }\cdot \gamma
_{_{\left\vert CA\right\vert }\left\vert CA\right\vert }}{\gamma
_{_{\left\vert A_{1}B\right\vert }\left\vert A_{1}B\right\vert }\cdot \gamma
_{_{\left\vert B_{1}C\right\vert }\left\vert B_{1}C\right\vert }\cdot \gamma
_{_{\left\vert C_{1}A\right\vert }\left\vert C_{1}A\right\vert }}\cdot \frac{%
\gamma _{_{\left\vert B_{1}A\right\vert }\left\vert B_{1}A\right\vert }\cdot
\gamma _{_{\left\vert C_{1}B\right\vert }\left\vert C_{1}B\right\vert }\cdot
\gamma _{_{\left\vert A_{1}C\right\vert }\left\vert A_{1}C\right\vert }}{%
\gamma _{_{\left\vert A_{1}B\right\vert }\left\vert A_{1}B\right\vert }\cdot
\gamma _{_{\left\vert B_{1}C\right\vert }\left\vert B_{1}C\right\vert }\cdot
\gamma _{_{\left\vert C_{1}A\right\vert }\left\vert C_{1}A\right\vert }}
\end{equation}%
From the relation (1) we have 
\begin{equation}
\frac{\gamma _{_{\left\vert B_{1}A\right\vert }\left\vert B_{1}A\right\vert
}\cdot \gamma _{_{\left\vert C_{1}B\right\vert }\left\vert C_{1}B\right\vert
}\cdot \gamma _{_{\left\vert A_{1}C\right\vert }\left\vert A_{1}C\right\vert
}}{\gamma _{_{\left\vert A_{1}B\right\vert }\left\vert A_{1}B\right\vert
}\cdot \gamma _{_{\left\vert B_{1}C\right\vert }\left\vert B_{1}C\right\vert
}\cdot \gamma _{_{\left\vert C_{1}A\right\vert }\left\vert C_{1}A\right\vert
}}=1,
\end{equation}%
so%
\begin{equation*}
\frac{\gamma _{_{\left\vert PA\right\vert }\left\vert PA\right\vert }}{%
\gamma _{_{\left\vert PA_{1}\right\vert }\left\vert PA_{1}\right\vert }}%
\cdot \frac{\gamma _{_{\left\vert PB\right\vert }\left\vert PB\right\vert }}{%
\gamma _{_{\left\vert PB_{1}\right\vert }\left\vert PB_{1}\right\vert }}%
\cdot \frac{\gamma _{_{\left\vert PC\right\vert }\left\vert PC\right\vert }}{%
\gamma _{_{\left\vert PC_{1}\right\vert }\left\vert PC_{1}\right\vert }}=%
\frac{\gamma _{_{\left\vert AB\right\vert }\left\vert AB\right\vert }\cdot
\gamma _{_{\left\vert BC\right\vert }\left\vert BC\right\vert }\cdot \gamma
_{_{\left\vert CA\right\vert }\left\vert CA\right\vert }}{\gamma
_{_{\left\vert AB_{1}\right\vert }\left\vert AB_{1}\right\vert }\cdot \gamma
_{_{\left\vert BC_{1}\right\vert }\left\vert BC_{1}\right\vert }\cdot \gamma
_{_{\left\vert CA_{1}\right\vert }\left\vert CA_{1}\right\vert }}. 
\end{equation*}
\end{proof}

\end{document}